\documentclass{article}[12pt]
\usepackage{url, graphicx, xcolor, hyperref, geometry}
\usepackage{latexsym,amsbsy}
\usepackage{lipsum}
\usepackage{amsfonts}
\usepackage{graphicx}
\usepackage{epstopdf}
\usepackage{algorithmic}
\usepackage{tikz}
\usetikzlibrary{matrix}
\usetikzlibrary{shapes,patterns,arrows,snakes,decorations.shapes}
\usetikzlibrary{positioning,fit,calc}
\usetikzlibrary{plotmarks}
\usepackage{amssymb,amsmath,amsthm}
\usepackage[obeyFinal]{easy-todo}
\usepackage{smartdiagram}
\usepackage{mathabx}
\usepackage{float}

\newtheorem{theorem}{Theorem}[section]
\newtheorem{lemma}{Lemma}[section]

\hypersetup{
    colorlinks,
    linkcolor={blue},
    citecolor={blue},
    urlcolor={blue}
}

\begin{document}

\title{A note on using the mass matrix as a preconditioner for the Poisson equation}

\author{Chen Greif\thanks{Department of Computer Science, The University of British Columbia, Vancouver, BC, V6T 1Z4, Canada. The work of first author  was supported in part by a Discovery Grant of the Natural Sciences and Engineering Research Council of Canada. \tt{greif@cs.ubc.ca},  \tt{yunhui.he@ubc.ca}. } 
\and Yunhui He\footnotemark[1]}

\maketitle
 
\begin{abstract}
 
We show that the mass matrix derived from finite elements can be effectively used as a preconditioner for iteratively solving the linear system arising from finite-difference discretization of the Poisson equation, using the conjugate gradient method. We derive  analytically the condition number of the preconditioned operator. Theoretical analysis shows that the ratio of the condition number of the  Laplacian to  the preconditioned operator  is $8/3$ in one dimension, $9/2$ in two dimensions,  and $2^9/3^4 \approx 6.3$ in three dimensions.  From this it follows that the expected iteration count for achieving a fixed reduction of the norm of the residual is smaller than a half of the number of the iterations of unpreconditioned CG in 2D and 3D. The scheme is easy to implement, and numerical experiments show its efficiency.
\end{abstract}

\section{Introduction}
\label{sec:intro}

Consider a standard  finite difference discretization of the Poisson equation in one, two  and three dimensions:
\begin{equation} 
-\Delta u = f 
\label{eq:poisson}
\end{equation}
on a simple domain $\Omega$, e.g., the unit interval, square or cube respectively, and subject to simple boundary conditions such as Dirichlet.
Suppose we discretize the problem on a uniform mesh whose size is $h =  \frac{1}{n+1}$, where $n$ is the number of meshpoints in a single direction of the domain.

The computational stencil for the Laplacian is given by 
\begin{equation}\label{eq:Laplace-stencil-1d}
  A_1= \frac{1}{h^2} \begin{bmatrix} -1 & 2 & -1 \end{bmatrix},
\end{equation}
\begin{equation}\label{eq:Laplace-stencil-2d}
A_2=\frac{1}{h^2}
  \begin{bmatrix}
               &           -1            &               \\
   -1 &        4         & -1 \\
               &   -1                     &
  \end{bmatrix},
\end{equation}
and
\begin{equation}\label{eq:Laplace-stencil-3d}
A_3=\frac{1}{h^2}
\begin{bmatrix}
  \begin{bmatrix}
 - 1
  \end{bmatrix}
  \quad 
  \begin{bmatrix}
               &           -1            &               \\
   -1 &        6        & -1 \\
               &   -1                     &
  \end{bmatrix}\quad 
  \begin{bmatrix}
  -1
  \end{bmatrix}
  \end{bmatrix}.
\end{equation}
For large problems in two and three dimensions we are interested in iterative methods, and specifically, in solving  the resulting linear system  by the Conjugate Gradient (CG) method. 

In our recent work \cite{CH2021addVanka}, we showed that in one dimension and in the context of multigrid, an element-wise Vanka smoother is equivalent to the scaled mass operator obtained from the linear finite element method, and in two dimensions, the  
element-wise Vanka smoother is equivalent to the scaled mass operator discretized by bilinear finite element method plus a scaled identity operator. While the context of that work is different, this has motivated us to ask whether the mass matrix obtained from the finite element method can be utilized as a preconditioner for the Laplacian. Here, we mean that preconditioning would amount to multiplying by the mass matrix; no inversion is involved. Such a possibility seems attractive given the ease of multiplying the Laplacian by the mass matrix, which is sparse and well conditioned. In this short note we provide analytical and numerical evidence  that  using the mass matrix in this manner at least doubles the convergence speed of CG in 2D and 3D at a modest computational cost.

In Section \ref{sec:convergence} we provide analytical observations on the condition number and the spectral distribution of the Laplacian scaled by the mass matrix, in comparison with the Laplacian. In Section \ref{sec:numerical} we validate our analysis experimentally. Brief concluding remarks are given in Section \ref{sec:conc}.

\section{Convergence analysis}
\label{sec:convergence} 

The stencil of the mass matrix  in 1D using linear finite elements is given by 
\begin{equation*}
M= \frac{h}{6}  \begin{bmatrix}
 1 & 4 & 1
\end{bmatrix}.
\end{equation*}
In the context of solving the Poission equation~\eqref{eq:poisson}, we consider the scaled mass matrix as a preconditioner for the Laplacian operator defined in \eqref{eq:Laplace-stencil-1d}, \eqref{eq:Laplace-stencil-2d} and \eqref{eq:Laplace-stencil-3d},  given by 
\begin{equation*}
M_1= h M,  \quad  M_2=  M \otimes M, \quad M_3=  h^{-1} M \otimes M\otimes M.
\end{equation*}

A well-known convergence bound on CG is determined by the condition number of the  coefficient matrix, and we will study the condition numbers of $A_d$ and  preconditioned operator $M_d A_d$. We rush to add that formally one would need to consider a symmetric positive definite similarity transformation of the latter, $M_d^{1/2} A_d M_d^{1/2} $, but the spectrum and the condition number are not affected by that transformation.

The following results are straightforward and/or well known, and are provided without a proof.
\begin{lemma}\cite{MR1807961}
The eigenvalues of $A_d, d=1,2,3,$ are given by
\begin{equation*}
\lambda(A_d) =  \frac{2}{h^2}   \sum_{j=1}^d \left (1-\cos(\pi h k_j) \right), \quad k_j \in \{ 1, 2,\ldots, n\}.
\end{equation*}
\end{lemma}

\begin{lemma}
The eigenvalues of $M_d, d=1,2,3,$ are given by
\begin{equation*}
\lambda(M_d) = \frac{h^2}{3^d}  \prod_{j=1}^d \left(2+\cos(\pi h k_j) \right), \quad k_j \in \{ 1, 2,\ldots, n\}.
\end{equation*}
\end{lemma}
\begin{lemma}\label{lemma:cond-Laplace}
The condition number of $A_d, d=1,2,3$, is given by
\begin{equation*}
\kappa = \frac{\lambda_{\rm max}(A_d) }{\lambda_{\rm min}(A_d) } = \frac{1-\cos(\pi h n)}{1-\cos(\pi h)}.
\end{equation*} 
\end{lemma}
From Lemma \ref{lemma:cond-Laplace}, we see that when $h\rightarrow 0$, the condition number $\kappa$ goes to infinity, which means that the iteration number will increase dramatically for CG without preconditioner.

We now consider the preconditioned operator $T_d = M_dA_d$ and analyze its eigenvalues.
\begin{theorem}
The eigenvalues of $T_d = M_dA_d$  are given by
\begin{equation}\label{eq:T_d-eigs}
\lambda(T_d) = \frac{2}{3^d} \prod_{j=1}^d \left(2+\cos(\pi h k_j) \right)  \left(\sum_{j=1}^d \left (1-\cos(\pi h k_j) \right) \right).
\end{equation}
Furthermore, the condition number of $T_d$ is  as follows:
\\
In 1D,
\begin{equation}\label{eq:kap-1d}
\kappa_p =   \frac{\left(2+\cos(\pi h [2/(3h)]) \right) \left(1-\cos(\pi h [2/(3h)]) \right) }{ (2+\cos(\pi h) ) (1-\cos(\pi h))},
\end{equation}
where $[\cdot]$ stands for  the integer part of a number.

\noindent In 2D
\begin{equation}\label{eq:kap-2d}
\kappa_p =   \frac{\left(2+\cos(\pi h [1/(2h)]) \right)^2  \left(1-\cos(\pi h [1/(2h)]) \right) }{\left(2+\cos(\pi h) \right)^2  \left(1-\cos(\pi h) \right) }.
\end{equation}
In 3D
\begin{equation}\label{eq:kap-3d}
\kappa_p =   \frac{\left(2+\cos(\pi h  \beta ) \right)^3  \left(1-\cos(\pi h \beta) \right) }{\left(2+\cos(\pi h) \right)^3  \left(1-\cos(\pi h) \right) },
\end{equation}
where $\beta=[{\rm arc cos}(1/4)/(\pi h)]$.
\end{theorem}
\begin{proof}
We can consider local Fourier analysis \cite{MR1807961} here to compute the eigenvalues of $M_d A_d$. When $M_d$ and $A_d$ are obtained from periodic operator, then the eigenvalues of 
$M_d A_d$ are the products of eigenvalues of $M_d$ and $A_d$.

When $d=1$, from \eqref{eq:T_d-eigs}, we have
\begin{equation*}
\lambda(T_1) =\frac{2}{3}  \left(2+\cos(\pi h k_j) \right)   \left (1-\cos(\pi h k_j) \right).
\end{equation*} 
Let us consider $f_1(x) =(2+x)(1-x)$ with $x \in[-1,1]$. Note that the maximum of $f_1(x)$  is achieved at $x=-\frac{1}{2}$ and the minimum is achieved at $x=1$. 
Thus, 
\begin{align*}
\lambda_{\max}(T_1)&=\frac{2}{3} \left(2+\cos(\pi h [2/(3h)]) \right) \left(1-\cos(\pi h [2/(3h)]) \right),\\
\lambda_{\min}(T_1) &=\frac{2}{3}(2+\cos(\pi h) ) (1-\cos(\pi h)),
\end{align*} 
which leads to \eqref{eq:kap-1d}.

When $d=2$, from \eqref{eq:T_d-eigs}, we have
\begin{equation*}
\lambda(T_2) =\frac{2}{9}  \left(2+\cos(\pi h k_1) \right)\left(2+\cos(\pi h k_2) \right)    \left (2-\cos(\pi h k_1) -\cos(\pi h k_2)\right).
\end{equation*} 
Let us consider $f_2(x,y) =(2+x)(2+y)(2-x-y)$ with $x,y \in[-1,1]$.  We compute the derivatives of $f_2$ with respect to $x$ and $y$, given by
\begin{align*}
f_{2,x}&= -4x-2y-2xy-y^2,\\
f_{2,y} &= -4y-2x-2xy-x^2.
\end{align*} 
Solving $f_{2,x}=f_{2,y}=0$ with $x,y\in[-1,1]$ gives $x=y=0$. It readily follows that $(0,0)$ is a local maximum point and $f_2(0,0)=8$. 
Next, we consider the boundary of $\Omega=[-1,1]^2$, and we find the extreme maximum are $f_2(1/2,-1)=f_2(-1,1/2)=25/4<f_2(0,0)$ and the minimum is  $f(1,1)=0$.
Thus, 
\begin{align*}
\lambda_{\max}(T_2)&=\frac{4}{9} \left(2+\cos(\pi h [1/(2h)]) \right)^2  \left(1-\cos(\pi h [1/(2h)]) \right),\\
\lambda_{\min}(T_2) &=\frac{4}{9}\left(2+\cos(\pi h) \right)^2  \left(1-\cos(\pi h) \right),
\end{align*} 
which leads to \eqref{eq:kap-2d}.
 
When $d=3$, from \eqref{eq:T_d-eigs}, we have
\begin{equation*}
\lambda(T_3) =\frac{2}{27}  \left(2+\cos(\pi h k_1) \right)\left(2+\cos(\pi h k_2) \right)  \left(2+\cos(\pi h k_3) \right)   \left (3-\cos(\pi h k_1) -\cos(\pi h k_2) -\cos(\pi hk_3)\right).
\end{equation*} 
Let us consider $f_3(x,y,z) =(2+x)(2+y)(2+z)(3-x-y-z)$ with $x,y,z \in[-1,1]$.  We compute the derivatives of $f_3$ with respect  to $x, y$ and $z$, given by
\begin{align*}
f_{3,x}&= (2+y)(2+z)(1-2x-y-z),\\
f_{3,y} &=(2+x)(2+z)(1-2y-x-z),\\
f_{3,z} &= (2+x)(2+y)(1-2z-x-y).
\end{align*} 
Solving $f_{3,x}=f_{3,y}=f_{3,z}=0$ with $x,y,z\in[-1,1]$ gives $x=y=z=1/4$. It is obvious that $(1/4,1/4,1/4)$ is a local maximum point and $f_3(1/4,1/4,1/4)=(9/4)^4$. 

Next, we consider $f_3(x,y,z)$ at the boundary of $\Omega=[-1,1]^3$, and  due to the symmetry  of $f_3$, we only need to consider two cases. 
One is that $z=1$ and $(x,y)\in[-1,1]^2$  and the other is $z=-1$ and $(x,y)\in[-1,1]^2$ . When $z=1$ and  $(x,y)\in[-1,1]^2$, $f_3(x,y,1)=g_1(x,y)=3(2+x)(2+y)(2-x-y)$ $(x,y)\in [-1,1]^2$. However,  from the proof of the case $d=2$, we know that the maximum of $g_1(x,y)$ is $g_1(0,0)=24$ and the minimum of $g_1(x,y)$ is $g_1(1,1)=0$.  When $z=-1$ and  $(x,y)\in[-1,1]^2$, $f_3(x,y,-1)=g_2(x,y)=(2+x)(2+y)(4-x-y)$. It can easily be shown that  the maximum of $g_2(x,y)$ is $g_2(1/2,1)=75/4$ and the minimum of $g_2(x,y)$ is $g_2(0,-1)=10$.  
Thus, the maximum of $f_3$ is $f_3(1/4,1/4,1/4)=(9/4)^4$ and the minimum is $f_3(1,1,1)=0$. This means that
\begin{align*}
\lambda_{\max}(T_3)&=\frac{2}{9} \left(2+\cos(\pi h \beta) \right)^3  \left(1-\cos(\pi h \beta) \right), \quad  \beta=[{\rm arc cos}(1/4)/(\pi h)],\\
\lambda_{\min}(T_3) &=\frac{2}{9}\left(2+\cos(\pi h) \right)^3  \left(1-\cos(\pi h) \right),
\end{align*} 
which yields \eqref{eq:kap-3d}.
\end{proof}
Next, we describe the relationship between the two condition numbers, $\kappa$ and $\kappa_p$.

\begin{theorem}\label{thm:ratio-results}
Define the ratio $r =\frac{\kappa}{\kappa_p} $. The the ratio satisfies: \\
In 1D
\begin{equation*}
r_1 = \frac{(1-\cos(\pi h n)  (2+\cos(\pi h) )}{ \left(2+\cos(\pi h [2/(3h)]) \right) \left(1-\cos(\pi h [2/(3h)]) \right) },
\end{equation*}
and
\begin{equation*}
\lim_{h\rightarrow 0} r_1 = \frac{8}{3}\approx 2.7.
\end{equation*} 
In 2D
\begin{equation*}
r_2 =  \frac{(1-\cos(\pi h n)  \left(2+\cos(\pi h) \right)^2} {\left(2+\cos(\pi h [1/(2h)]) \right)^2  \left(1-\cos(\pi h [1/(2h)]) \right) },
\end{equation*}
and 
\begin{equation*}
\lim_{h\rightarrow 0} r_2 = \frac{9}{2}=4.5.
\end{equation*} 
In 3D
\begin{equation*}
r_3 =  \frac{(1-\cos(\pi h n)  \left(2+\cos(\pi h) \right)^3} {\left(2+\cos(\pi h \beta ) \right)^3  \left(1-\cos(\pi h \beta) \right) },
\end{equation*}
where $\beta=[{\rm arc cos}(1/4)/(\pi h)]$ and 
\begin{equation*}
\lim_{h\rightarrow 0} r_3 = \frac{2^9}{3^4} \approx 6.3.
\end{equation*} 
\end{theorem}
 
%
From Theorem \ref{thm:ratio-results} it is interesting  to notice that the gains in terms of condition number ratios grow with the dimension; this suggests that our approach is particularly effective for 3D.

It is well known that the convergence bound of CG satisfies \cite{saad2003iterative}  
\begin{equation*}
\| x-x_k  \|_A  \leq 2  \left(\frac{ \sqrt{ \kappa }-1 } { \sqrt{ \kappa}+1} \right)^k \| x-x_0\|_A.
\end{equation*} 

Requiring  $\| x-x_k  \|_A \leq  \epsilon$ gives
\begin{equation*}
 k  \approx \frac{ 1} {2} {\rm log } (\epsilon/2) \sqrt{ \kappa }.
\end{equation*} 
It follows that to achieve the same convergence tolerance $\epsilon$, the ratio of iteration numbers of CG without preconditioner to that of CG with preconditioner is
\begin{equation}\label{eq:theoretical-itn-ratio}
\frac{k}{k_p} = \sqrt{\frac{\kappa}{\kappa_p}} =\sqrt{r}.
\end{equation}
%
%
\section{Numerical experiments}
\label{sec:numerical}

To demonstrate the efficiency of the mass matrix as a preconditioner for the Laplacian,  we  consider the Poisson equation in two and three dimensions on the unit square and cube, respectively, subject to homogeneous Dirichlet boundary conditions. We discretize it using a uniform mesh, as briefly described in Section~\ref{sec:intro}.  We run CG with and without preconditioner and stop the iteration when the residual norm  is below $10^{-8}$.

In Figure~\ref{fig:eigenvalues-2d-cg} we illustrate the effect that preconditioning has on the eigenvalues of the matrix. It is evident that most of the eigenvalues of the preconditioned matrix have values relatively close to 1, which explains the effectiveness of the preconditioner.
\begin{figure}[H] 
\begin{center}
  \includegraphics[width=0.99\textwidth]{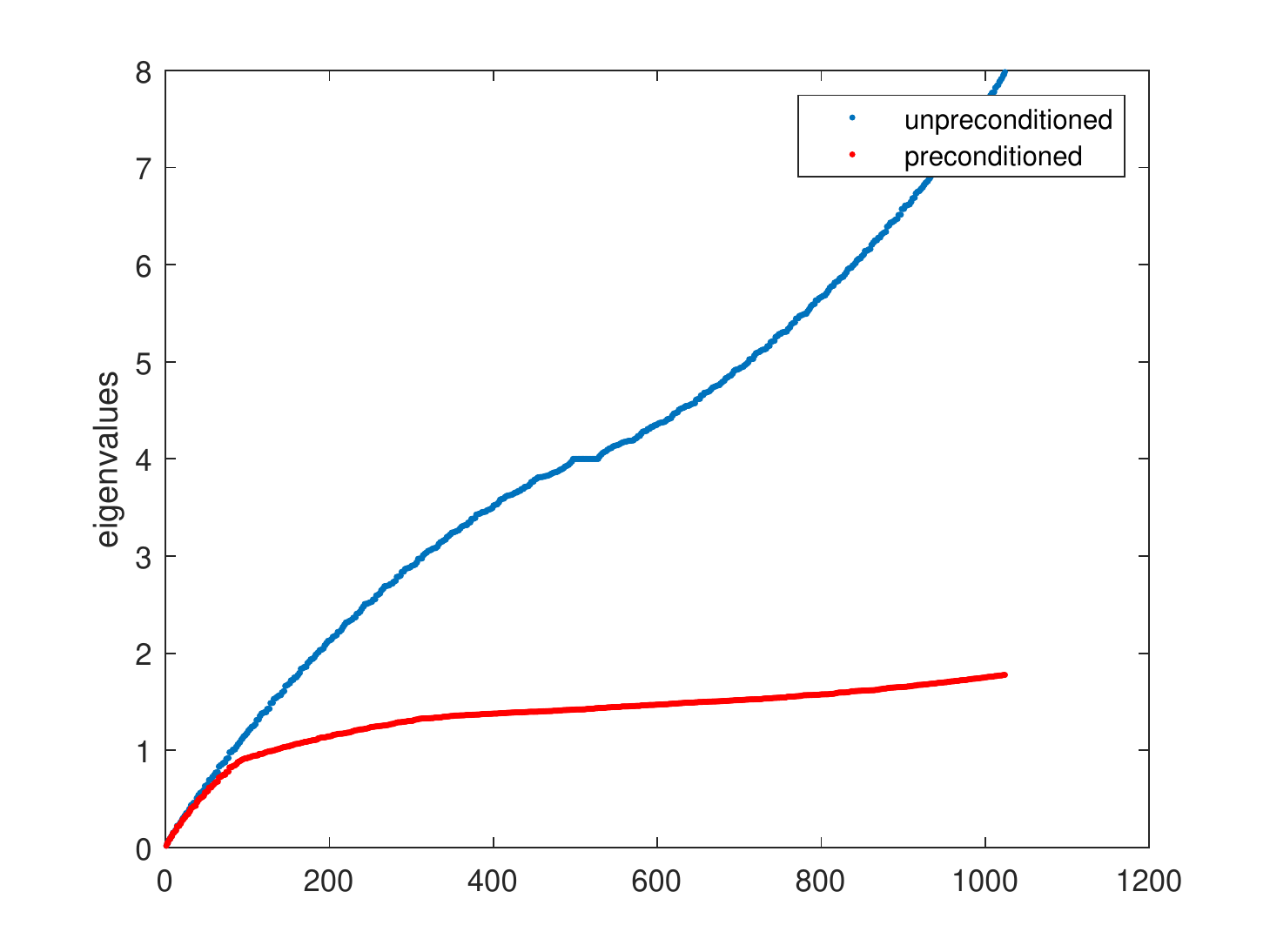}
   \caption{Eigenvalues of the Laplacian vs. the product of the mass matrix by the Laplacian. This is a 2D problem with $n=32$, i.e., the matrix is of dimensions $1,024 \times 1,024$.}\label{fig:eigenvalues-2d-cg}
   \end{center}
 \end{figure}

In Table \ref{tab:ratio}, we numerically compute the ratio $\frac{\kappa}{\kappa_p}$ for different values of meshgrid size $h$ and dimension $d$.  The results in Table \ref{tab:ratio} very closely match the analytical results in Theorem \ref{thm:ratio-results}. For fixed $d$,  when $h$ increases, the ratio $r_d$ increases, but $r_d$ is bounded.
\begin{table}[H]
 \caption{The  ratio of two condition numbers}
\centering
\begin{tabular}{c l  l l }
\hline
$r_d$  &8     &16              &32  \\
\hline
 \hline
$r_1$                  &32.1634/12.6914 $\approx$ 2.5          &  116.4612/44.2414$\approx$ 2.6        & 440.6886/165.8836    $\approx$ 2.7              \\
\hline
$r_2$                  & 32.1634/7.6173 $\approx$ 4.2          & 116.4612/26.3451$\approx$ 4.4        & 440.6886/98.3943  $\approx$ 4.5                \\
\hline
$r_3$                  & 32.1634/5.5393  $\approx$ 5.8         & 116.4612/18.8900$\approx$ 6.2        & 440.6886/70.1771 $\approx$ 6.3                \\
\hline
\end{tabular}\label{tab:ratio}
\end{table}

We now move to show some convergence results in 2D and 3D. In Table~\ref{tab:iterations} we summarize our findings. 
\begin{table}[H]
\caption{Iteration counts for 2D and 3D experiments. The column under `mtx-size' gives the sizes of the linear systems considered (number of degrees of freedom). The term `th-itn-ratio' stands for `theoretical iteration counts  ratio' (see \eqref{eq:theoretical-itn-ratio}), as explained in the example. The term `itn-ratio' refers to the ratio between iteration counts for the unpreconditioned case (`itn-unprec') and iteration counts for the preconditioned case (`itn-prec').}
\centering
\begin{tabular}{ccccccc}
\hline
Type      & $n$  & mtx-size            &itn-unprec   & itn-prec  & th-itn-ratio & itn-ratio\\
\hline
 \hline
2D      & 32  & 1,024  & 62 & 30    & 2.12 & 2.07            \\  
       & 64   & 4,096  & 122 & 58    & & 2.10            \\   
           & 128   & 16,384  & 231 & 110    & & 2.10            \\  
                    & 256   & 65.536  & 454 & 215    & & 2.11            \\  \hline \hline  
3D       & 32     &  32,768 &   81 & 33 & 2.51 & 2.45     \\  
       &  64    &262,144    &  158 & 63  & & 2.51        \\  
         &  96    &884,736    &  225 & 90  & & 2.50        \\  
           &  128    & 2,097,152    &  296 & 118  & & 2.51        \\  \hline
\end{tabular}\label{tab:iterations}
\end{table}

These results in Table~\ref{tab:iterations}  are consistent with our theoretical findings. 
When $d=2$, $r_2\approx 4.5$ and $\sqrt{4.5}\approx 2.12$. Thus,  unpreconditioned CG is expected to take approximately 2.12 times the iteration number of preconditioned CG.
When $d=3$, $r_3\approx 6.3$ and $\sqrt{6.3}\approx 2.51$. Thus,  unpreconditioned CG is expected to take approximately 2.51 times the iteration number of preconditioned CG. The table shows that those predictions are remarkably accurate in both 2D and 3D.

In Figure \ref{fig:Errors-2d-cg}, we show convergence history for $n=128$ and $n=256$ in 2D. In Figure \ref{fig:Errors-3d-cg} we show convergence history $n=64$ and $n=128$ in 3D.

  \begin{figure}[H]
 \includegraphics[width=0.49\textwidth]{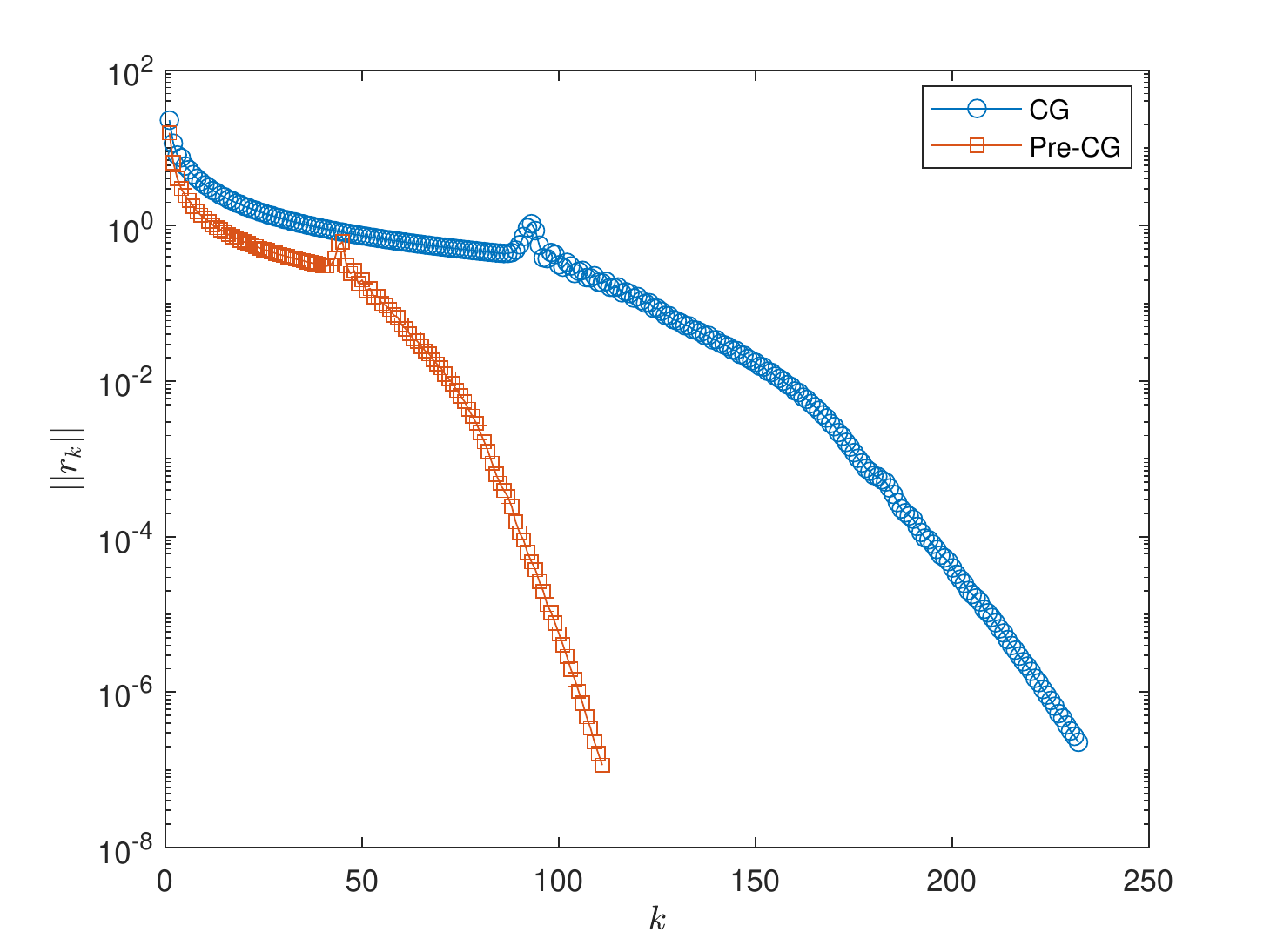}
\includegraphics[width=0.49\textwidth]{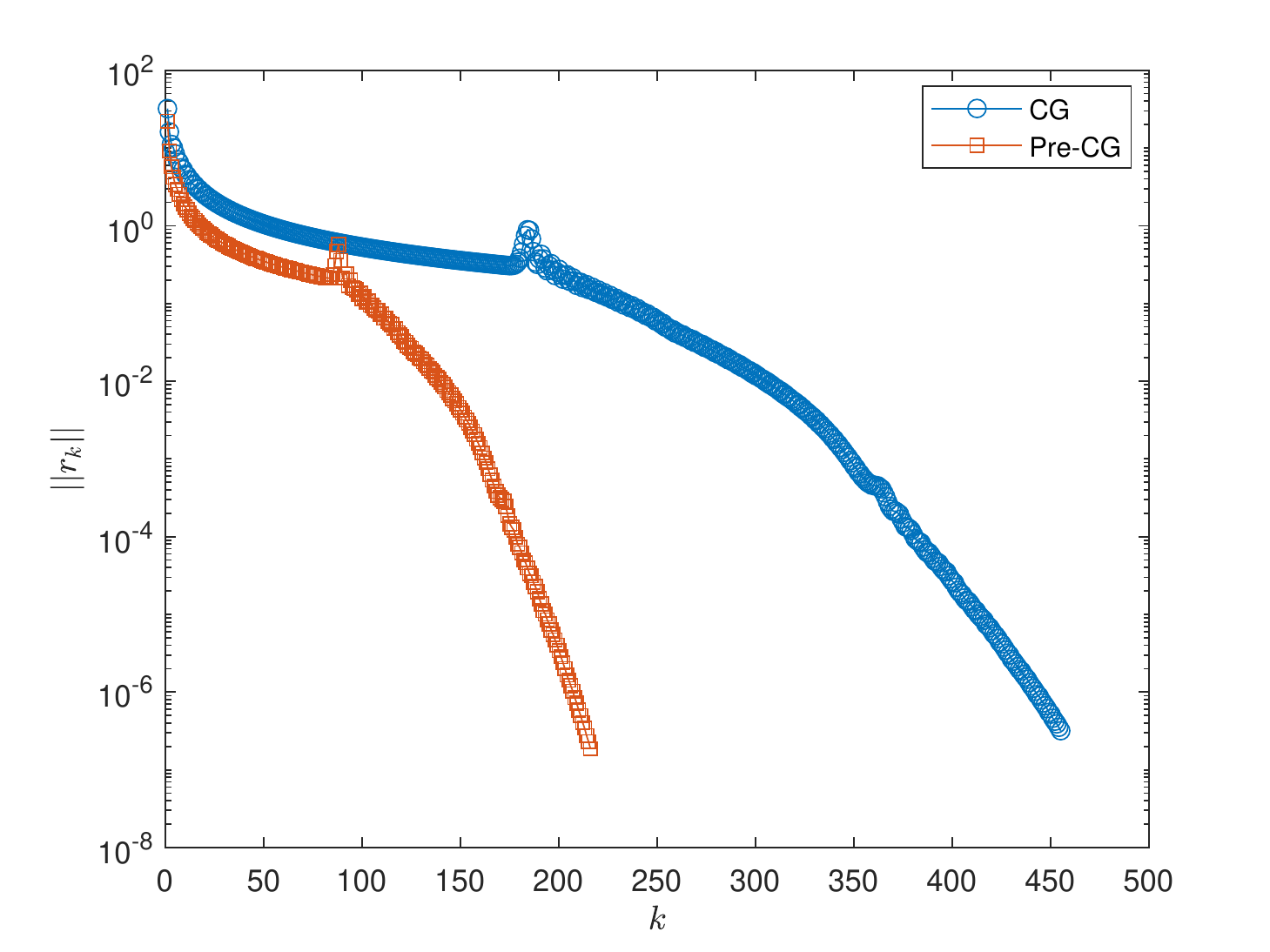}
   \caption{Convergence history of CG with and without a preconditioner in 2D.  Left:   $n=128$. Right: $n=256$.}\label{fig:Errors-2d-cg}
 \end{figure}  %
 
\begin{figure}[H] 
  \includegraphics[width=0.49\textwidth]{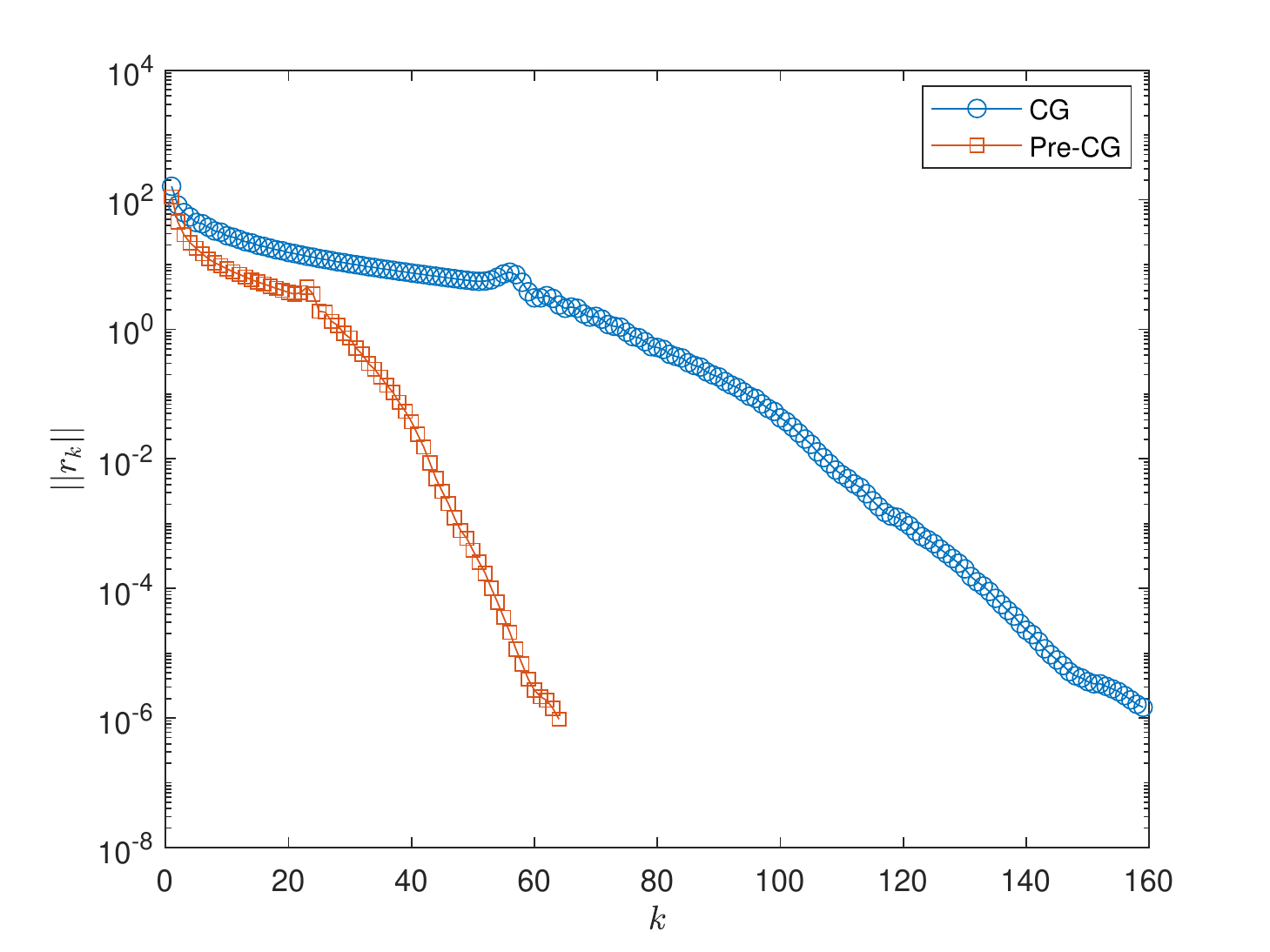}
\includegraphics[width=0.49\textwidth]{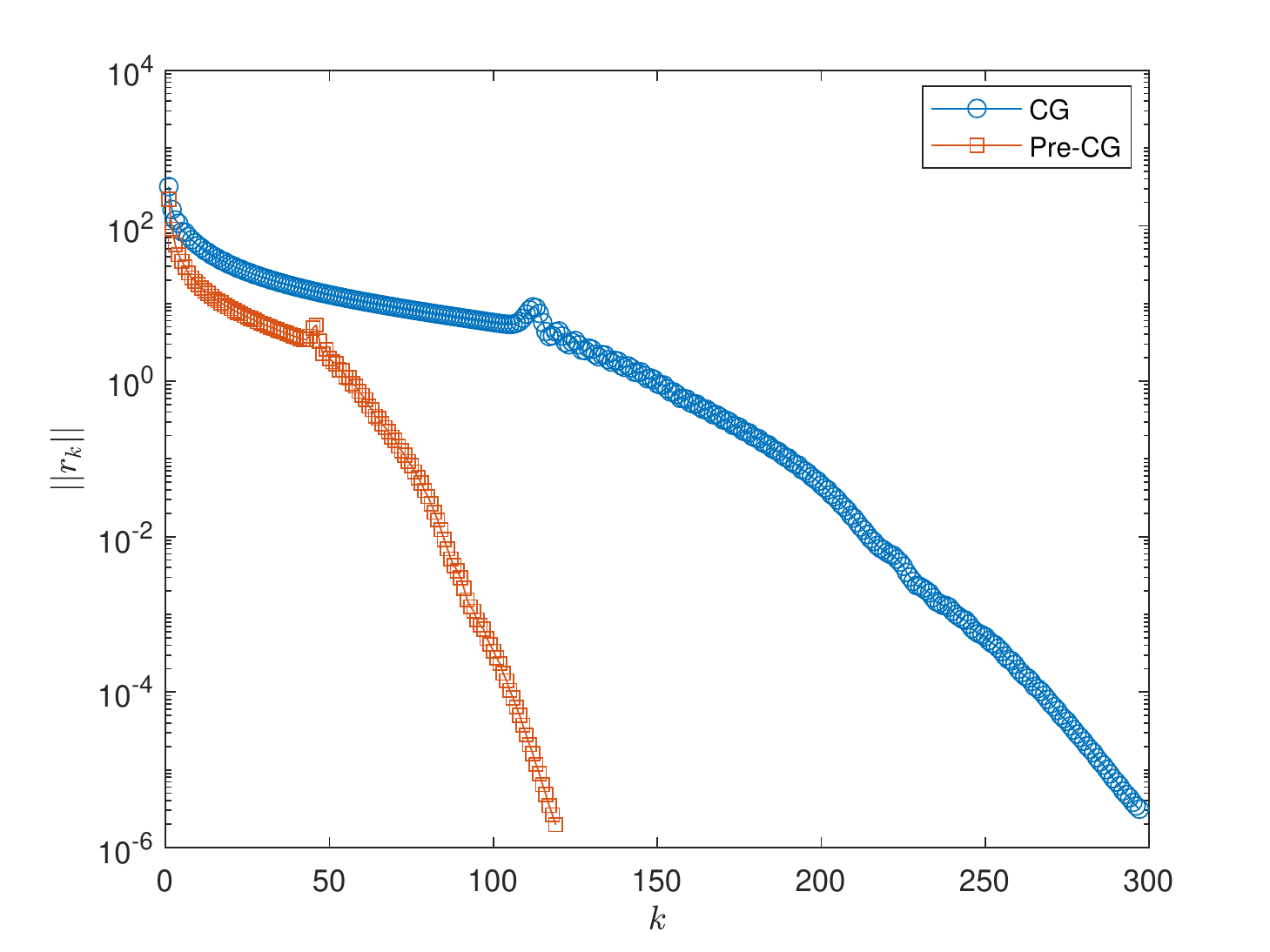}
   \caption{Convergence history of CG with and without preconditioner in 3D.  Left:   $n=64$. Right: $n=128$. }\label{fig:Errors-3d-cg}
 \end{figure}

\section{Concluding remarks}
\label{sec:conc}

Our analytical results  provide a remarkably accurate estimate of the condition number and iteration counts for CG. At a minimal cost that amounts to a matrix-vector product by the sparse and well-conditioned mass matrix, convergence speed is at least doubled. The gains are stronger in the 3D case. The cost of the additional matrix-vector product per iteration is modest, especially if considered in a parallel computing environment. Therefore, the overall computational gains are meaningful.

The proposed scheme is extremely simple and easy to implement and may make it possible to utilize the mass matrix in other problems in potentially useful ways.

\bibliographystyle{siam}
\bibliography{refKry}
\end{document}